\documentclass[12pt]{amsart}
\usepackage[alphabetic]{amsrefs}
\usepackage{graphicx}
\usepackage{amscd}
\usepackage{upgreek}
\usepackage{stmaryrd}
\usepackage{longtable}
\usepackage[T1]{fontenc}
\usepackage{latexsym, amsmath, amssymb, amsthm}
\usepackage{rsfso}
\usepackage[mathscr]{eucal}
\usepackage{mathtools}
\usepackage{mathptmx}
\usepackage{titletoc}
\usepackage{wrapfig}
\usepackage{float}
\usepackage{xypic}
\usepackage{microtype}
\usepackage{dsfont}
\usepackage[svgnames]{xcolor}
\usepackage{xcolor}
\usepackage{color}
\usepackage[colorlinks = true,
            linkcolor  = DarkBlue,
            urlcolor   = DarkRed,
            citecolor  = DarkGreen]{hyperref}
\usepackage{hyperref}
\allowdisplaybreaks	
\usepackage[centering, includeheadfoot, hmargin=0.95in, tmargin=0.6in,
  bmargin=0.6in, headheight=6pt]{geometry}
\newtheorem{theorem}{Theorem}[section]
\newtheorem{prop}[theorem]{Proposition}

\newtheorem{lem}[theorem]{Lemma}
\newtheorem{coro}[theorem]{Corollary}

\newtheorem{thm}[theorem]{Theorem}

\newtheorem{rem}[theorem]{Remark}
\newtheorem{exam}[theorem]{Example}

\newcommand{\ideal}[1]{\ensuremath{\left\langle #1 \right\rangle}}

\DeclareMathOperator{\Aut}{Aut}
\DeclareMathOperator{\Der}{Der}

\DeclareMathOperator{\GL}{GL}
\DeclareMathOperator{\Hom}{Hom}
\DeclareMathOperator{\HLie}{HLie}

\newcommand{\C}{\mathbb{C}}

\newcommand{\Z}{\mathbb{Z}}
\newcommand{\N}{\mathbb{N}}

\newcommand{\g}{\mathfrak{g}}
\newcommand{\h}{\mathfrak{h}}
\newcommand{\p}{\mathfrak{p}}
\newcommand{\uu}{\mathfrak{u}}
\newcommand{\gl}{\mathfrak{gl}}
\newcommand{\ssl}{\mathfrak{sl}}
\newcommand{\MLie}{{\rm HLie}_m}
\newcommand{\RLie}{{\rm HLie}_r}
\newcommand{\ILie}{{\rm HLie}_i}

\newcommand{\ra}{\longrightarrow}

\newcommand{\B}{\mathcal{B}}
\newcommand{\HH}{\mathcal{H}}
\newcommand{\V}{\mathcal{V}}

\newcommand{\hbo}{$\hfill\Diamond$}

\begin{document}
\title{A Commutative Algebra Approach to Multiplicative \\Hom-Lie Algebras}
\def\shorttitle{A Commutative Algebra Method to Multiplicative Hom-Lie Algebras}

\author{Yin Chen}
\address{School of Mathematics and Statistics, Northeast Normal University, Changchun, China \& Department of Mathematics and Statistics, Queen's University, Kingston, K7L 3N6, Canada}
\email{ychen@nenu.edu.cn}

\author{\scshape Runxuan Zhang}
\address{School of Mathematics and Statistics, Northeast Normal University, Changchun, China}
\email{zhangrx728@nenu.edu.cn}

\begin{abstract}
Let $\mathfrak{g}$ be a finite-dimensional complex Lie algebra and $\textrm{HLie}_{m}(\mathfrak{g})$ be the affine variety of all multiplicative Hom-Lie algebras on $\mathfrak{g}$.
We use a method of computational ideal theory to describe $\textrm{HLie}_{m}(\mathfrak{gl}_{n}(\mathbb{C}))$, showing that $\textrm{HLie}_{m}(\mathfrak{gl}_{2}(\mathbb{C}))$ consists of two 1-dimensional and one 3-dimensional irreducible components and 
$\textrm{HLie}_{m}(\mathfrak{gl}_{n}(\mathbb{C}))=\{\textrm{diag}\{\delta,\dots,\delta,a\}\mid \delta=1\textrm{ or }0,a\in\mathbb{C}\}$ for $n\geqslant 3$. We construct a new family of multiplicative Hom-Lie algebras on the Heisenberg Lie algebra $\mathfrak{h}_{2n+1}(\mathbb{C})$ and characterize the affine varieties $\textrm{HLie}_{m}(\mathfrak{u}_{2}(\mathbb{C}))$ and $\textrm{HLie}_{m}(\mathfrak{u}_{3}(\mathbb{C}))$.
We also study the derivation algebra $\textrm{Der}_{D}(\mathfrak{g})$ of a multiplicative Hom-Lie algebra
$D$ on $\mathfrak{g}$ and, under some hypotheses on $D$, we prove that the Hilbert series $\mathcal{H}(\textrm{Der}_{D}(\mathfrak{g}),t)$ is a rational function.
\end{abstract}

\date{\today}
\thanks{2020 \emph{Mathematics Subject Classification}. 17B61; 13P10; 13P15.}
\keywords{Hom-Lie algebra; general linear Lie algebra; Heisenberg Lie algebra.}
\maketitle \baselineskip=15.8pt

\dottedcontents{section}[1.16cm]{}{1.8em}{5pt}
\dottedcontents{subsection}[2.00cm]{}{2.7em}{5pt}

\section{Introduction}
\setcounter{equation}{0}
\renewcommand{\theequation}
{1.\arabic{equation}}
\setcounter{theorem}{0}
\renewcommand{\thetheorem}
{1.\arabic{theorem}}

\noindent In the last fifteen years, Hom-algebra structures have occupied an important place in
nonassociative algebras, deformation theory and mathematical physics.
Realizing Hom-Lie algebra structures on a vector space has substantial ramifications in the study of representation theory,
deformations of infinite-dimensional Lie algebras and generalized Yang-Baxter equations, whereas
finding a powerful method to describe these Hom-Lie algebras is indispensable in developing efficient classifying tools.
Jin-Li's Theorem \cite{JL08}*{Proposition 2.1}, proving that all Hom-Lie algebras on complex simple finite-dimensional
Lie algebras except for $\ssl_{2}(\C)$ are trivial, serves as a motivational example.
Our primary objective is to describe multiplicative Hom-Lie algebra structures on several typical families of complex finite-dimensional Lie algebras and our approach depends upon techniques from commutative algebra.

Motivated by characterizing algebraic structures of some $q$-deformations of the Witt and the Virasoro algebras,
\cite{HLS06} originally introduced the notion of a Hom-Lie algebra (on a vector space), showing that
these $q$-deformations have a Hom-Lie algebra structure. This initial definition of a Hom-Lie algebra was also modified slightly to the current version; see  \cite{MS08}, \cite{BM14} and \cite{She12}.
Recently, the structure and representation theory of Hom-Lie algebras, Hom-associative, and even Hom-Novikov algebras,  have been studied extensively; see for example \cite{HLS06}, \cite{MS08}, \cite{Yau11},  \cite{ZHB11} and references therein.
We concentrate on Hom-Lie algebra structures on a finite-dimensional complex Lie algebra because
the well-developed structure theory of Lie algebras and related representation theory have been demonstrated to be useful
in solving such  problems; see \cite{Bau99}.

Let $\g$ be a finite-dimensional complex Lie algebra. A linear transformation $D$ on $\g$ is called a \textit{Hom-Lie algebra} structure on $\g$ if the Hom-Jacobi identity: $[D(x),[y,z]]+[D(y),[z,x]]+[D(z),[x,y]]=0$
holds for all $x,y,z\in \g$. A Hom-Lie algebra $D$ on $\g$ is said to be \textit{multiplicative} if $D$ is a Lie algebra homomorphism.
Inspired by \cite{JL08}, we wonder whether there exists a nontrivial (multiplicative) Hom-Lie algebra structure on non-semisimple complex Lie algebras and further, if there exist such Hom-Lie algebras, we also seek a systematic way to describe
and classify them up to isomorphism. Consolidating and comparing with  existing methods (see \cite{Rem18} and \cite{GDSSV20}), we take a point of view of affine varieties on the set of all Hom-Lie algebras and
multiplicative Hom-Lie algebras on $\g$. This means that techniques from computational ideal theory will be our main source of tools.

\subsection*{Affine varieties of Hom-Lie algebras}
We use $\dim_{\C}(*)$ and $\dim(*)$ to denote
the dimension and the Krull dimension of $*$ as a $\C$-vector space and an affine variety over the complex field $\C$, respectively.
Suppose $\dim_\C(\g)=n$ and $M_n(\C)$ denotes the affine space of all $n\times n$-matrices over $\C$.
With respect to a chosen basis  $\{e_1,e_2,\dots,e_n\}$ of $\g$, each element of $M_n(\C)$ corresponds to a linear
transformation on $\g$. The main objects of study in the present paper are the vector space
$\HLie(\g):=\{D\in M_n(\C)\mid D \textrm{ is a Hom-Lie algebra on }\g\}$ and the affine variety
$\MLie(\g):=\{D\in \HLie(\g)\mid D \textrm{ is multiplicative}\}.$
We also refer to an element $D\in \MLie(\g)$ as a \textit{regular} Hom-Lie algebra on $\g$ if it is an automorphism of $\g$,
and we refer to $D$ as \textit{involutive} if it is an involution.
Let $\RLie(\g)$ and $\ILie(\g)$ denote the subsets of all regular Hom-Lie algebras and  of all involutive Hom-Lie algebras on
$\g$ respectively. The following set inclusions hold:
\begin{equation}
\label{eq1.1}
\ILie(\g)\subseteq \RLie(\g)\subseteq \MLie(\g)\subseteq \HLie(\g).
\end{equation}
We observe that
a Hom-Lie algebra structure $D\in\HLie(\g)$ can be determined by  finitely many polynomial equations in at most $n^2$ variables, which means that $\HLie(\g), \MLie(\g), \RLie(\g)$ and $\ILie(\g)$ all can be viewed as affine varieties in an affine space of dimension at most $n^2$. Moreover, it is easy to see that $\HLie(\g)$ is a linear variety and thus it is
irreducible; $\MLie(\g), \RLie(\g)$ and $\ILie(\g)$ are not  linear in general and thus they may not be irreducible.
We also note that $\RLie(\g)\subseteq \GL(\g)\cap\Aut(\g)$ is a linear algebraic group (not necessarily irreducible), where
$\GL(\g)$ and $\Aut(\g)$ denote the general linear group on $\g$ and the automorphism group of $\g$ respectively.

\subsection*{Main results} Specifically, we study multiplicative Hom-Lie algebra structures on three important families of
 finite-dimensional  complex Lie algebras: the general linear Lie algebra $\gl_{n}(\C)(n\geqslant 2)$,
 the Heisenberg Lie algebra $\h_{2n+1}(\C)(n\geqslant 1)$  and the Lie algebra $\uu_{n}(\C)(n\geqslant 2)$ of upper triangular matrices, which are the most typical examples of reductive, nilpotent  and solvable Lie algebras respectively.

Theorem \ref{thm2.10} is our first theorem that describes the geometric structure of the affine variety
$\MLie(\gl_{2}(\C))$, showing that $\textrm{HLie}_{m}(\mathfrak{gl}_{2}(\mathbb{C}))$ consists of two 1-dimensional and one 3-dimensional irreducible components. The key to our proof is to apply a Gr\"{o}bner basis method from computational ideal theory in commutative algebra; compared with \cite{XJL15}*{Corollary 3.4} for the case of $\ssl_{2}(\C)$.
Our second result establishes a complete description of $\MLie(\gl_{n}(\C))$ for $n\geqslant 3$. Based on  Jin-Li's Theorem on
$\MLie(\ssl_{n}(\C))$ and the close relationship between $\ssl_{n}(\C)$ and  $\gl_{n}(\C)$, we prove that
if $D\in \MLie(\gl_{n}(\C))$, then $D$ must be equal to a diagonal matrix $\textrm{diag}\{\delta,\dots,\delta,a\}$
where $\delta$ is either 1 or 0 and $a\in\C$; see Theorem \ref{thm3.3} for details. The two results also positively answer
the previous question on the existence of nontrivial Hom-Lie algebra structures on $\gl_{n}(\C)$.

Hom-Lie algebra structures on a nilpotent or solvable Lie algebra are more complicated than that on a semi-simple or reductive Lie algebra. Proposition \ref{prop4.2}, as our third result, gives a new family of multiplicative Hom-Lie algebras on $\h_{2n+1}(\C)$.
As applications we prove that all the three containments appeared in (\ref{eq1.1}) are strict
for the case $\g=\h_{2n+1}(\C)$; see Corollaries \ref{coro4.3}--\ref{coro4.5}.
Our final major result is about the derivation algebra $\Der_{D}(\g)$ of a Hom-Lie algebra $D$ on a Lie algebra $\g$ and
the corresponding Hilbert series $\HH(\Der_{D}(\g),t)$. The derivation algebra of a Hom-Lie algebra was introduced and studied
by \cite{She12}, aiming at developing the representation and cohomology theory of Hom-Lie algebras.
Inspired by the classical topic on the rationality of a Hilbert series \cite[Theorem 3.3.1]{NS02} and compared with \cite[Theorem 1.4]{CCZ21},
we prove in Theorem \ref{thm5.5} that
under some additional hypotheses, $\HH(\Der_{D}(\g),t)$ is a rational function.

\subsection*{Organization}
Section \ref{sec2} is devoted to describing  three irreducible components of  the affine variety $\MLie(\gl_2(\C))$,
particularly demonstrating that there exists a nontrivial Hom-Lie algebra structure on $\gl_2(\C)$.
Section \ref{sec3} completely characterizes multiplicative Hom-Lie algebra structures on $\gl_{n}(\C)$ for $n\geqslant 3$.
In Section \ref{sec4}, we study Hom-Lie algebra structures on the Heisenberg Lie algebra $\h_{2n+1}(\C)$ and
the upper triangular Lie algebra $\uu_{n}(\C)$. Section \ref{sec5} is mainly concentrated on
the rationality of the Hilbert series of the derivation algebra of Hom-Lie algebras.

\subsection*{Acknowledgements} This work was partially supported by NNSF of China (No. 11301061).
The authors would like to thank the two referees and the editor for their helpful suggestions and comments.

\section{The Affine Variety $\MLie(\gl_2(\C))$}\label{sec2}
\setcounter{equation}{0}
\renewcommand{\theequation}
{2.\arabic{equation}}
\setcounter{theorem}{0}
\renewcommand{\thetheorem}
{2.\arabic{theorem}}

\noindent After summarizing the fundamental facts on $\gl_2(\C)$ and some preparations for the vanishing ideal of the affine variety $\MLie(\gl_2(\C))$, we capitalize on a Gr\"{o}bner basis method from computational ideal theory to analyze the geometric structure of  $\MLie(\gl_2(\C))$.

\subsection*{Basics for $\gl_2(\C)$}
We start with some basic facts on  $\gl_2(\C)$. For $1\leqslant i,j\leqslant 2$, let $E_{ij}$ be the $2\times 2$-matrix
 in which the $(i,j)$-entry is 1 and 0 otherwise.
For notational brevity, we set $e_{1}:=E_{11}, e_{2}:=E_{12},e_{3}:=E_{21}$ and $e_{4}:=E_{22}$. Then $\{e_{1},\dots,e_{4}\}$
is the standard basis of  $\gl_2(\C)$, subject to the following nontrivial relations:
\begin{equation}\label{relgl2}
[e_1,e_2]=e_2, [e_1,e_3]=-e_3, [e_2,e_3]=e_1-e_4,[e_2,e_4]=e_2,[e_3,e_4]=-e_3.
\end{equation}
Let $D\in\MLie(\gl_2(\C))$ be an arbitrary element. With respect to this standard basis of  $\gl_2(\C)$, we always identify $D$ with
a matrix $(a_{ij})_{4\times 4}^{T}$ in $M_{4}(\C)$. This means that
\begin{equation}\label{actgl2}
D(e_{i})=\sum_{j=1}^{4} a_{ij}e_{j}
\end{equation}
for $i=1,\dots,4$.

\subsection*{The vanishing ideal of $\MLie(\gl_2(\C))$}
Let $A:=\C[x_{ij}\mid 1\leq i,j\leq 4]$ be the polynomial ring in 16 variables. Then $A$ can be viewed as the coordinate ring of the  affine space $M_{4}(\C)$ of all $4\times 4$-matrices over $\C$ in the natural way. To articulate the vanishing ideal of $\MLie(\gl_2(\C))$, we need to define the following 23 polynomials in $A$:
\begin{eqnarray*}
f_1:= x_{11}-x_{44},&&
f_{2}:=x_{23}x_{32}+x_{41}x_{44}-x_{42}x_{43}-(x_{41}^2+x_{41}+x_{44}^2-x_{44})/2,\\
f_3:= x_{12}+x_{42},&& f_4:= x_{23}x_{41} - x_{23}x_{44} - x_{23} + 2x_{43}^2,\\
f_{5}:= x_{13}+x_{43},&& f_{6}:= x_{23}x_{42} - (x_{41}x_{43} -x_{43}x_{44} +x_{43})/2,\\
f_{7}:=x_{14} - x_{41},&&f_{8}:= x_{32}x_{41} - x_{32}x_{44} - x_{32} + 2x_{42}^2,\\
f_{9}:=x_{21} + x_{43},&&f_{10}:= x_{32}x_{43} - (x_{41}x_{42} - x_{42}x_{44} +x_{42})/2,\\
f_{11}:= x_{22} - x_{33},&&f_{12}:= x_{33}^2+ x_{41}x_{44} - x_{42}x_{43} - (x_{41}^2  -x_{41}  +x_{44}^2 +x_{44})/2,\\
f_{13}:= x_{23}x_{33}+x_{43}^2,&&f_{14}:= x_{33}x_{41} - x_{33}x_{44} + x_{33} + 2x_{42}x_{43},\\
f_{15}:=x_{24} - x_{43},&&f_{16}:= x_{33}x_{42} -(x_{41}x_{42} - x_{42}x_{44} -x_{42})/2,\\
f_{17}:= x_{31} + x_{42},&&f_{18}:= x_{33}x_{43} - (x_{41}x_{43} - x_{43}x_{44} -x_{43})/2,\\
f_{19}:= x_{32}x_{33} + x_{42}^2,&&f_{20}:=x_{42}(x_{41}^2 - 2x_{41}x_{44} + 4x_{42}x_{43} + x_{44}^2 - 1),\\
f_{21}:= x_{34} - x_{42},&&f_{22}:=x_{43}(x_{41}^2 - 2x_{41}x_{44} + 4x_{42}x_{43} + x_{44}^2 - 1),
\end{eqnarray*} and $f_{23}:= x_{41}^3 - 3x_{41}^2x_{44} + 4x_{41}x_{42}x_{43} + 3x_{41}x_{44}^2 - x_{41} - 4x_{42}x_{43}x_{44} - x_{44}^3 + x_{44}.$ Throughout this section we let $I$ be the ideal generated by $\{f_i\mid 1\leq i\leq 23\}$ in $A$, and we will show that $I$ is exactly the vanishing ideal of $\MLie(\gl_2(\C))$.  One might be interested in how to obtain these polynomials $f_i$. Indeed, the way of constructing these $f_i$ is quite direct by using the definition of multiplicative Hom-Lie algebra, i.e., choosing a generic matrix $D$ in $\MLie(\gl_2(\C))$, the properties of algebraic homomorphism and Hom-Jacobi 
identity lead to a bunch of polynomial equations in the entries of $D$. By deleting redundant equations, one will reveal the above 23 polynomials $f_i$ such that $f_i(D)=0$ for all $i$; see the proof of Lemma \ref{lem2.1} below. For any ideal $J$ of $A$, we also define $\V(J):=\{T\in M_{4}(\C)\mid f(T)=0, \textrm{ for all }f\in J\}$. Note that $\V(J)=\V(\sqrt{J})$.

\begin{lem}\label{lem2.1}
$\MLie(\gl_2(\C))\subseteq\V(I)$.
\end{lem}

\begin{proof}
Suppose $D=(a_{ij})^{T}_{4\times 4}\in \MLie(\gl_2(\C))$ is an arbitrary element. It suffices to show that $D\in\V(I)$; equivalently the valuation of every $f_i$ at $D$ is zero for $1\leqslant i\leqslant 23$. Indeed,  since $D$ is an algebraic homomorphism and it satisfies the Hom-Jacobi identity, we see that $D([e_i,e_j])=[D(e_i),D(e_j)]$ and $[D(e_i),[e_j,e_k]]+[D(e_j),[e_k,e_i]]+[D(e_k),[e_i,e_j]]=0$, for all $1\leq i,j,k\leq 4$. Putting these equations together, we can use the relations among $e_1,\dots,e_4$ in (\ref{relgl2}) and the rule in (\ref{actgl2}) to solve them. Deleting redundant equations, we finally obtain 23 polynomial equations in $a_{ij}$.  These  $f_i$ are the corresponding polynomials of the equations. Hence, $f_i(D)=0$ for all $i$. 
\end{proof}

\begin{rem}\label{rem2.2}
{\rm
In fact, we will show that the two affine varieties $\MLie(\gl_2(\C))$ and $\V(I)$ are equal. By Lemma \ref{lem2.1}, it suffices to show that $\V(I)$ is contained in $\MLie(\gl_2(\C))$. To achieve this, we need to investigate the geometric structure of $\V(I)$. More precisely, our first step is to find all irreducible components $\V(\p_1),\dots,\V(\p_k)$ of $\V(I)$ for some $k\in\N^+$, and our last step is to show that every element in each component $\V(\p_i)$ is a multiplicative Hom-Lie algebra structure on
$\gl_{2}(\C)$, where all $\p_{i}$ are prime ideals of $A$. \hbo
}\end{rem}

\subsection*{Irreducible components of $\V(I)$}
We will see that the affine variety $\V(I)$ has three irreducible components: $\V(\p_1),\V(\p_2),\V(\p_3)$. To better understand  these components, we need to define the following six auxiliary polynomials in $A$:
\begin{eqnarray*}
\alpha:=x_{14}-x_{44}, &\beta:=x_{41}-x_{44},& h:=\beta^{2}+\beta.\\
g_1:=x_{22}-(\beta-1)/2,& g_2:=x_{23}x_{32}+x_{42}x_{43}-(\beta+1)/2,&
 g_{3}:=\beta^2+4x_{42}x_{43}-1.
\end{eqnarray*}
We also define four ideals of $A$ as follows:
\begin{eqnarray*}
\p_{1} & := & (f_1,\alpha,\beta,x_{12},x_{13},x_{21},x_{22},x_{23},x_{24},x_{31},x_{32},x_{33},x_{34},x_{42},x_{43}),\\
\p_{2} & := & (f_1,\alpha+1,\beta+1,x_{12},x_{13},x_{21},x_{22}-1,x_{23},x_{24},x_{31},x_{32},x_{33}-1,x_{34},
x_{42},x_{43}),\\
\p_{3} & := &(f_{1},f_{3},f_{4},\dots,f_{10},f_{15},f_{17},f_{21},g_1,g_{2},g_{3}),\\
\p&:=&(f_{1},f_{7},h,x_{12},x_{13},x_{21},x_{22}+\beta,x_{23},x_{24},x_{31},x_{32},x_{33}+\beta,x_{34},x_{42},x_{43}).
\end{eqnarray*}

\begin{lem}\label{lem2.3}
The ideals  $\p_1,\p_2$ and $\p_3$ are prime.
\end{lem}

\begin{proof}
Clearly, $\p_1$ and $\p_2$ are prime ideals, since the generators of $\p_1$ and $\p_2$ are polynomials of degree 1. To show that $\p_3$ is prime, it suffices to show that
$A/\p_3$ is an integral domain. In fact,  $A/\p_3\cong \C[x_{23},x_{32},x_{41},\dots,x_{44}]/(g_2,g_3,f_{2i}\mid 2\leq i\leq 5)$. The latter is isomorphic to $$\Big(\C[\beta,x_{23},x_{32},x_{42},x_{43}]/(g_2,g_3,f_{2i}\mid 2\leq i\leq 5)\Big)[x_{44}].$$ Thus it suffices to show that $B:=\C[\beta,x_{23},x_{32},x_{42},x_{43}]/(g_2,g_3,f_{2i}\mid 2\leq i\leq 5)$ is an integral domain.

We claim that the image of $\beta-1$ in $B$ is not a zero-divisor. Let $J_1:=(g_2,g_3,f_{2i}\mid 2\leq i\leq 5)$ and $J_2:=(\beta-1)$ be ideals of $\C[\beta,x_{23},x_{32},x_{42},x_{43}]$.
The feasibility of this claim is equivalent to deciding whether the colon ideal $(J_1:J_2)$ is equal to $J_1$. By
\cite{DK15}*{Section 1.2.4} we see that $(J_1:J_2)=(\beta-1)^{-1}(J_1\cap J_2).$
We use the lexicographic ordering with $\beta>x_{23}>x_{32}>x_{42}>x_{43}$ in $\C[\beta,x_{23},x_{32},x_{42},x_{43}]$.
Applying \cite[Algorithm 1.1.9]{DK15}, a direct calculation shows that the following eight polynomials
\begin{eqnarray*}
\beta - 2x_{23}x_{32} - 2x_{42}x_{43} + 1, &  & x_{23}^2x_{32} + x_{23}x_{42}x_{43} - x_{23} + x_{43}^2, \\
 x_{23}^2x_{42} - x_{23}x_{43} + x_{43}^3, && x_{23}x_{32}^2 + x_{32}x_{42}x_{43} - x_{32} + x_{42}^2,\\
 x_{23}x_{32}x_{42} - x_{32}x_{43} + x_{42}^2x_{43}, && x_{23}x_{32}x_{43} - x_{23}x_{42} + x_{42}x_{43}^2,\\
 x_{23}x_{42}^2 - x_{32}x_{43}^2,&& x_{32}^2x_{43} - x_{32}x_{42} + x_{42}^3
\end{eqnarray*}
form a Gr\"{o}bner basis for the ideal $J_{1}$. Denote this Gr\"{o}bner basis by $G_{1}$. To derive a Gr\"{o}bner basis for $J_{1}\cap J_{2}$,
we consider the polynomial ring $\C[t,\beta,x_{23},x_{32},x_{42},x_{43}]$ and use the lexicographic ordering with $t>\beta>x_{23}>x_{32}>x_{42}>x_{43}$. Let $J_{12}$ be the ideal of $\C[t,\beta,x_{23},x_{32},x_{42},x_{43}]$ generated by
$(1-t)\cdot J_{1}+t\cdot J_{2},$
where the products are formed by multiplying each generator of $J_{1}$ and $J_{2}$ by $1-t$ and $t$ respectively. It follows from \cite[Corollary 2.1.1]{Vas98} that
$J_{1}\cap J_{2}=J_{12}\cap \C[\beta,x_{23},x_{32},x_{42},x_{43}].$
Applying  \cite[Algorithm 1.1.9]{DK15} again we observe  that  these polynomials:
\begin{eqnarray*}
&&t\beta - t,~~ tx_{23}x_{32} - t + (\beta x_{23}x_{32} + \beta x_{42}x_{43} + \beta - 3 x_{23}x_{32} - x_{42}x_{43} + 1)/2,\\
&&    tx_{23}x_{42} - tx_{43} + \beta x_{43}/2 - x_{23}x_{42} + x_{43}/2,~~ tx_{32}x_{43} - tx_{42} + \beta x_{42}/2 - x_{32}x_{43} + x_{42}/2,\\
  &&   tx_{42}^2 - \beta x_{32}/2 + x_{32}/2 - x_{42}^2,   ~~  tx_{43}^2 - \beta x_{23}/2 + x_{23}/2 - x_{43}^2,\\
      && t x_{42}x_{43} - (\beta x_{23} x_{32} +\beta x_{42}x_{43} - x_{23}x_{32} +x_{42}x_{43})/2
\end{eqnarray*}
together with $(\beta-1)\cdot G_{1}$, form a Gr\"{o}bner basis for $J_{12}$.
Thus \cite[Algorithm 1.2.1]{DK15}  implies that $(\beta-1)\cdot G_{1}$ is a Gr\"{o}bner basis for $J_{1}\cap J_{2}$. This means that $(J_{1}:J_{2})$ can be generated by $G_{1}$ and so $(J_1:J_2)=J_1$, the claim follows.

Hence, $B$ can be embedded  into $B^*:=B[\frac{1}{\beta-1}]$ and we need only to show that $B^*$ is an integral domain. Since the images of $x_{23}$ and $x_{32}$ in $B$ can be expressed by the images of  $x_{42}$, $x_{43}$ and $(\beta-1)^{-1}$, it follows that $B^*\cong (\C[\beta,x_{42},x_{43}]/(\beta^2+4x_{42}x_{43}-1)[\frac{1}{\beta-1}]$. As $\beta^2+4x_{42}x_{43}-1$ is irreducible, $B^*$ is isomorphic to a localization of an integral domain. Thus $B^*$ is also an integral domain.
The proof is completed.
\end{proof}

Lemma \ref{lem2.3} has the following immediate consequences.

\begin{coro}\label{coro2.4}
Let $C_{a}:=\begin{pmatrix}
           a& 0&0& a\\
           0& 0&0& 0\\
            0& 0&0& 0\\
            a& 0&0& a\\
         \end{pmatrix}$ and $D_{a}:=\begin{pmatrix}
           a& 0&0& a-1\\
           0& 1&0& 0\\
            0& 0&1& 0\\
            a-1& 0&0& a\\
         \end{pmatrix}$ with $a\in\C$.  Then $\V(\p_{1})=\{C_{a}\mid a\in \C\}$ and $\V(\p_{2})=\{D_{a}\mid a\in \C\}$ are 1-dimensional irreducible affine varieties.
\end{coro}

\begin{coro}\label{coro2.5}
The affine variety $\V(\p_{3})$ is 3-dimensional and irreducible, consisting of all  matrices
$$E_{a,b,c,\xi}:=\begin{pmatrix}
           a-\xi& -c&-b& a\\
           -b& \frac{\xi-1}{2}&\frac{-2c^{2}}{\xi-1}& b\\
            -c& \frac{-2b^{2}}{\xi-1}& \frac{\xi-1}{2}& c\\
            a& c&b& a-\xi\\
         \end{pmatrix},$$
 with $a,b,c,1\neq\xi\in\C$ and $\xi^{2}+4bc-1=0$.
\end{coro}

\begin{proof}
We have seen from Lemma \ref{lem2.3} that $\p_3$ is a prime ideal, thus the variety  $\V(\p_{3})$ is irreducible. For a 
generic element $E\in \V(\p_{3})$, the fact that  generators of $\p_3$ evaluated on $E$ are zero, shows that $E$
must have the form $E_{a,b,c,\xi}$ for some $a,b,c,1\neq\xi\in\C$ with the condition $\xi^{2}+4bc-1=0$. Hence, $\V(\p_{3})$ is consisting of all such matrices $E_{a,b,c,\xi}$. This fact also shows that the coordinate ring of $\V(\p_{3})$ is isomorphic to $\C[x,y,z,w]/(z^2+4xy-1)$, which has Krull dimension 3. 
\end{proof}

\begin{lem}\label{lem2.6}
$\V(\p_1)\cup\V(\p_2)\cup\V(\p_3)\subseteq\V(I)$.
\end{lem}

\begin{proof}
It suffices to show that $I\subseteq \p_1$, $I\subseteq \p_2$ and $I\subseteq \p_3$. To show the first containment, it is easy to see that all generators of $I$ except for $f_{2},f_{7},f_{12},f_{23}$ are zero modulo $\p_1$. Thus it is sufficient to show that $f_{2},f_{7},f_{12},f_{23}$ are equal to zero modulo $\p_1$. In fact, working over modulo $\p_1$, we see that $f_2\equiv f_{12}\equiv -\frac{1}{2}(\beta^2+\beta)\equiv 0$; $f_7=\alpha-\beta\equiv 0$; $f_{23}\equiv\beta^3-\beta\equiv 0.$
This proves that $I\subseteq \p_1$. Similarly, one can show that $I\subseteq \p_2$ and $I\subseteq \p_3$.
\end{proof}

\begin{lem}\label{lem2.7}
 $\p_{1}\cdot\p_{2}\subseteq\p.$
\end{lem}

\begin{proof}
By \cite[Proposition 6, page 185]{CLO07} we see that the set of products of any two elements that come from the generating sets of $\p_1$ and $\p_2$ respectively can generate the product $\p_1\cdot\p_2$. Hence, we only need to show that every element in $\B_1\cdot\B_2:=\{b_1b_2\mid b_i\in \B_i,1\leq i\leq 2\}$ belongs to $\p$, where
$$ \B_1 := \{\alpha,\beta,x_{22},x_{33}\} \textrm{ and }\B_2 := \{\alpha+1,\beta+1,x_{22}-1,x_{33}-1\}.$$
As  the three elements $f_{7},x_{22}+\beta,x_{33}+\beta$ are contained in $\p$, we see that  $\alpha\equiv\beta$ and $x_{22}\equiv x_{33}\equiv -\beta$ modulo $\p$. This indicates that it is sufficient to show that $\beta\cdot\B_{2}$ is contained in $\p$.
Note that $h\in\p$, so $\beta(\alpha+1)\equiv\beta(\beta+1)=h\equiv 0$ and $\beta(x_{22}-1)\equiv\beta(x_{33}-1)\equiv -h\equiv 0$ modulo $\p$. 
\end{proof}

\begin{lem}\label{lem2.8}
$\V(I)\subseteq\V(\p_1)\cup\V(\p_2)\cup\V(\p_3)$.
\end{lem}

\begin{proof}
Lemma \ref{lem2.7}, together with the fact that  $\V(\p_1)\cup\V(\p_2)\cup\V(\p_3)=\V(\p_1\cdot\p_2\cdot\p_3)$, implies  that it suffices to show that the product $\p \cdot\p_3$ is contained in $I$. By \cite[Proposition 6, page 185]{CLO07}, we only need to show that every element in $\B\cdot \B_3:=\{bb_3\mid b\in \B\textrm{ and }b_{3}\in\B_{3}\}$ belongs to $I$, where
$\B := \{h,x_{12},x_{13},x_{21},x_{22}+\beta,x_{23},x_{24},x_{31},x_{32},x_{33}+\beta,x_{34},x_{42},x_{43}\}$ and
 $\B_3 := \{g_{1},g_{2},g_{3}\}.$
Moreover, as $f_{3},f_{5},f_{9},f_{11},f_{15},f_{17},f_{21}\in I$, the set $\B$ can be replaced by $\B':=
\{h,x_{23},x_{32},x_{33}+\beta,x_{42},x_{43}\}.$
Now we have to verify that each element of $\B_{3}\cdot\B'$ belongs to $I$.
Throughout the rest of the proof, we are working over modulo $I$. By the definition of the generators $f_{i}$ of $I$, we observe that
\begin{eqnarray*}
x_{23}x_{32}-x_{42}x_{43}-(\beta^{2}+\beta)/2=0,&&
x_{23}(\beta-1) + 2x_{43}^2=0,\\
x_{23}x_{42} - x_{43}(\beta+1)/2=0,&&
x_{32}(\beta-1) + 2x_{42}^2=0,\\
x_{32}x_{43} - x_{42}(\beta+1)/2=0,&&
x_{33}^2 - x_{42}x_{43} - (\beta^{2}-\beta)/2=0,\\
x_{33}(\beta+1)+ 2x_{42}x_{43}=0,&&
x_{33}x_{42} -x_{42}(\beta-1)/2=0,\\
x_{33}x_{43} - x_{43}(\beta-1)/2=0,&&
x_{42}g_{3}=0,\\
x_{43}g_{3}=0,&&
\beta g_{3}=0.
\end{eqnarray*}
We will use these equations to complete the proof.
Note that $g_{1}h=(x_{33}-(\beta-1)/2)(\beta^{2}+\beta)=(x_{33}(\beta+1)-(\beta^{2}-1)/2)\beta=(-2x_{42}x_{43}-(\beta^{2}-1)/2)\beta=-(\beta g_{3})/2=0$; $g_{1}x_{23}=(x_{33}-(\beta-1)/2)x_{23}=x_{23}x_{33}+x_{43}^{2}=f_{13}=0$;
$g_{1}x_{32}=f_{19}=0$; $g_{1}(x_{33}+\beta)=x_{33}^{2}+(\beta+1)x_{33}/2-(\beta^{2}-\beta)/2=(x_{33}(\beta+1)+ 2x_{42}x_{43})/2=0$; $g_{1}x_{42}=x_{33}x_{42} -x_{42}(\beta-1)/2=0$; and $g_{1}x_{43}=x_{33}x_{43} - x_{43}(\beta-1)/2=0$.
This shows that $g_{1}\B'\subseteq I$. Further, $g_{2}=2x_{42}x_{43}+(\beta^{2}-1)/2=-(\beta+1)g_{1}$
and $g_{3}=(\beta^{2}-1)-2x_{33}(\beta+1)=2(\beta+1)g_{1}$. This implies that $g_{2}\cdot\B'$ and $g_{3}\cdot\B'$ are contained in $I$. Hence, $\B_{3}\cdot\B'\subseteq I$ and we are done.
\end{proof}

Lemmas \ref{lem2.6} and \ref{lem2.8} combine to obtain

\begin{coro}\label{coro2.9}
$\V(I)=\V(\p_1)\cup\V(\p_2)\cup\V(\p_3)$.
\end{coro}

\begin{thm}\label{thm2.10}
$\MLie(\gl_2(\C))=\V(I)=\V(\p_1)\cup\V(\p_2)\cup\V(\p_3)$.
\end{thm}

\begin{proof}
We have seen in Lemma \ref{lem2.1} that $\MLie(\gl_2(\C))\subseteq\V(I)$. To complete the proof, by Corollary \ref{coro2.9},
it suffices to show that any $C_{a},D_{a}$ and $E_{a,b,c,\xi}$ appeared in Corollaries \ref{coro2.4} and \ref{coro2.5} belong to $\MLie(\gl_2(\C))$. One can verify the assertion by a direct calculation.
\end{proof}

\begin{rem}
{\rm
Note that each irreducible component we obtained in Theorem \ref{thm2.10} is the closure of infinite families of algebras, and 
thus there are no rigid algebras in the variety $\MLie(\gl_2(\C))$; see for example \cite[Chapter 5]{GK96} for more details about rigid algebras. 
\hbo
}\end{rem}

\section{The Affine Varieties $\MLie(\gl_n(\C))~~(n\geqslant 3)$} \label{sec3}
\setcounter{equation}{0}
\renewcommand{\theequation}
{3.\arabic{equation}}
\setcounter{theorem}{0}
\renewcommand{\thetheorem}
{3.\arabic{theorem}}

\noindent This section characterizes multiplicative Hom-Lie algebra structures on $\gl_{n}(\C)$ for all $n\geqslant 3$.
Suppose that $\{e_{1},e_{2},\dots,e_{n^{2}-1}\}$ is a basis of $\ssl_{n}(\C)$. By the classical fact that a matrix that commutes with all matrices must be a scalar matrix, we see that the center of $\gl_{n}(\C)$ consists of all scalar matrices. We may take a nonzero scalar matrix $z$ in $\gl_{n}(\C)$, and then $\gl_{n}(\C)=\ssl_{n}(\C)\oplus \C\cdot z$. Moreover, we have the following useful fact.

\begin{lem}\label{lem3.1}
$[\gl_{n}(\C),\gl_{n}(\C)]=\ssl_{n}(\C)$.
\end{lem}

\begin{proof}
Note that 
$[\gl_{n}(\C),\gl_{n}(\C)]=[\ssl_{n}(\C)\oplus \C\cdot z,\ssl_{n}(\C)\oplus \C\cdot z]=[\ssl_{n}(\C),\ssl_{n}(\C)]$, which is a nonzero  ideal of $\ssl_{n}(\C)$. However, $\ssl_{n}(\C)$ is a simple Lie algebra, thus $[\ssl_{n}(\C),\ssl_{n}(\C)]=\ssl_{n}(\C)$.
\end{proof}

\begin{coro}\label{coro3.2}
Every homomorphism on $\gl_{n}(\C)$ restricts to a homomorphism on $\ssl_{n}(\C)$.
\end{coro}

\begin{proof}
Let $D$ be a homomorphism from $\gl_{n}(\C)$ to itself. It suffices to show that $\ssl_{n}(\C)$ is stable under the action of $D$.
Indeed, given an element $x\in \ssl_{n}(\C)$, by Lemma \ref{lem3.1}, we may write $x=\sum_{\textrm{finite}} [x_{i},x_{j}]$ for some
$x_{i},x_{j}\in \gl_{n}(\C)$. Thus $D(x)=\sum_{\textrm{finite}} D([x_{i},x_{j}])=\sum_{\textrm{finite}} [D(x_{i}),D(x_{j})]\in [\gl_{n}(\C),\gl_{n}(\C)]=\ssl_{n}(\C)$.
\end{proof}

We are ready to prove our second result.

\begin{thm}\label{thm3.3}
Let $n\geqslant 3$ and $D\in\MLie(\gl_n(\C))$ be an arbitrary element. Then with respect to the basis $\{e_{1},e_{2},\dots,e_{n^{2}-1},z\}$, $D$ is equal to either ${\rm diag}\{1,\dots,1,a\}$ or ${\rm diag}\{0,\dots,0,a\}$, where $a\in\C$.
\end{thm}

\begin{proof}
By \cite[Corollary 3.4 (ii)]{XJL15} we see that $\MLie(\ssl_n(\C))$ consists of the identity matrix $I_{n^{2}-1}$
and the zero matrix.  On the other hand, Corollary \ref{coro3.2} implies that $D$ restricts to an element of $\MLie(\ssl_n(\C))$. Thus the restriction of $D$ on $\ssl_n(\C)$ is either equal to $I_{n^{2}-1}$ or 0.

For the first case, we have $D(x)=x$ for $x\in\ssl_n(\C)$. Note that for $i=1,2,\dots,n^{2}-1$, we have $0=D(0)=D([e_{i},z])=[D(e_{i}),D(z)]=[e_{i},D(z)]$ and moreover, $[z,D(z)]=0$. This implies that $D(z)$ is also an element of the center of
$\gl_n(\C)$. Thus $D(z)=az$ for some $a\in\C$. Hence, with respect to the basis $\{e_{1},e_{2},\dots,e_{n^{2}-1},z\}$, we see that $D={\rm diag}\{1,\dots,1,a\}$.

For the second case, we have $D(x)=0$ for $x\in\ssl_n(\C)$. Suppose $D(z)=\sum_{i=1}^{n^{2}-1} a_{i}e_{i}+az$, where
all $a_{i}$ and $a$ belong to $\C$. For $x\in \ssl_n(\C)$, as $\ssl_n(\C)=[\ssl_n(\C),\ssl_n(\C)]$, we may write
$x=\sum_{i,j=1}^{n^{2}-1}a_{ij}[e_{i},e_{j}]$ for  $a_{ij}\in\C$. Then
\begin{eqnarray*}
[D(z),x] & = & \left[D(z),\sum_{i,j=1}^{n^{2}-1}a_{ij}[e_{i},e_{j}]\right]\\
 & = & \sum_{i,j=1}^{n^{2}-1}a_{ij}[D(z),[e_{i},e_{j}]]\quad (\textrm{by the Hom-Jacobi identity})\\
 &=&   -\sum_{i,j=1}^{n^{2}-1}a_{ij}([D(e_{i}),[e_{j},z]]+[D(e_{j}),[z,e_{i}]])=0.
\end{eqnarray*}
This fact, together with $[D(z),z]=0$, implies that $D(z)$ is in the center of $\gl_n(\C)$. Thus $D(z)=az$ for some $a\in\C$.
This shows that in this case, $D={\rm diag}\{0,\dots,0,a\}$.
\end{proof}

\begin{rem}
{\rm
Note that a Hom-Lie algebra structure on a vector space is determined by two parts: a linear map $D$ and 
the bracket product $[-,-]$. The two parts both are required to satisfy some conditions. In the present paper, we fix the part of bracket product on the vector space, thus if we embed $\MLie(\gl_n(\C))$ into the whole variety  $\MLie(\C^{n^2})$ of 
$n^2$-dimensional multiplicative complex Hom-Lie algebras, Theorem \ref{thm3.3} says that $\MLie(\gl_n(\C))$ is the union of two lines in $\MLie(\C^{n^2})$ when $n\geqslant 3$.
\hbo
}\end{rem}

\section{The Affine Varieties $\MLie(\h_{2n+1}(\C))$ and $\MLie(\uu_{n}(\C))$}\label{sec4}
\setcounter{equation}{0}
\renewcommand{\theequation}
{4.\arabic{equation}}
\setcounter{theorem}{0}
\renewcommand{\thetheorem}
{4.\arabic{theorem}}

\noindent In this section, we study multiplicative Hom-Lie algebra structures on the Heisenberg Lie algebra and the Lie algebra of upper triangular matrices  which are the most typical examples of nilpotent  and solvable Lie algebras respectively.

\subsection*{Heisenberg Lie algebras}
Let $n\in\N^{+}$. Recall that the complex Heisenberg Lie algebra $\h_{2n+1}(\C)$  of dimension $2n+1$
generated by the following $(n+2)\times (n+2)$-matrices
$$x_{i}=\begin{pmatrix}
      0&e_{i}^{T}&  0  \\
      0&  0_{n}&0\\
      0&0&0
\end{pmatrix}, y_{j}=\begin{pmatrix}
      0&0&  0  \\
      0&  0_{n}&e_{j}\\
      0&0&0
\end{pmatrix},z=\begin{pmatrix}
      0&0&  1  \\
      0&  0_{n}&0\\
      0&0&0
\end{pmatrix}$$
where $1\leqslant i,j\leqslant n$, $\{e_{1},\dots,e_{n}\}$ is the standard basis of $\C^{n}$, $e_{i}^{T}$ denotes
the transpose of $e_{i}$ and $0_{n}$ denotes the zero matrix of size $n$.
In $\h_{2n+1}(\C)$, there are the following relations:
\begin{equation}
\label{eq4.1}
[x_{i},y_{j}]=\delta_{ij}\cdot z, [x_{i},z]=0=[y_{j},z].
\end{equation}

We have already known a complete characterization of multiplicative Hom-Lie algebras on the 3-dimensional complex Heisenberg Lie algebra $\h_{3}(\C)$; see for example \cite[Corollary 2.3]{AC19}.

\begin{prop}\label{prop4.1}
The affine variety $\MLie(\h_3(\C))$ is a 6-dimensional irreducible affine variety, consisting of the following matrices:
$$\begin{pmatrix}
     bf-ce &a&d    \\
      0 &b&e    \\
        0 &c&f
\end{pmatrix}$$
where $a,b,c,d,e,f\in\C$. In particular, there exists a nontrivial involutive Hom-Lie algebra structure on $\h_{3}(\C)$.
\end{prop}

\begin{proof}
See \cite[Corollary 2.3]{AC19} for a proof of the first assertion. More description on 3-dimensional Hom-Lie algebras can be found in \cite{GDSSV20} and \cite{Rem18}. For the second assertion, we take $a=c=d=e=0, b=f=-1$ and
$$D=\begin{pmatrix}
      1&0&0    \\
     0 &-1&0\\
     0&0&-1
\end{pmatrix}.$$ Then $D\in \MLie(\h_3(\C))$ gives rise to a nontrivial involutive Hom-Lie structure on $\h_{3}(\C)$.
\end{proof}

Giving a complete description of multiplicative Hom-Lie algebra structures on $\h_{2n+1}(\C)$ is a difficult and challenging task.
Here we construct a family of multiplicative Hom-Lie algebra structures on $\h_{2n+1}(\C)$, generalizing
the construction in Proposition \ref{prop4.1}; and we will see that this new construction leads to some remarkable consequences.

\begin{prop}\label{prop4.2}
Let $D(a,b,c,d;\alpha):=\begin{pmatrix}
    ad-bc  & \alpha  \\
      0&\Theta
\end{pmatrix}$ be the block matrix of size $2n+1$, where $\alpha=(a_{1},b_{1},\dots,a_{n},b_{n})$, $\Theta={\rm diag}\{\underbrace{\theta,\dots,\theta}_{n}\}$ and
$\theta=\begin{pmatrix}
      a& b   \\
      c& d
\end{pmatrix}$. Then $$D(a,b,c,d;\alpha)\in \MLie(\h_{2n+1}(\C)).$$
\end{prop}

\begin{proof}
For simplicity, we denote $D(a,b,c,d;\alpha)$ by $D$ throughout the proof.
Then $D$ acts on the basis $\{z,x_{1},y_{1},x_{2},y_{2},\dots,x_{n},y_{n}\}$ of  $\h_{2n+1}(\C)$ as follows:
$$D(z)=\det(\theta)\cdot z;\quad D(x_{i})=a_{i}z+ax_{i}+cy_{i}; \quad D(y_{i})=b_{i}z+bx_{i}+dy_{i}$$
for $1\leqslant i\leqslant n$.  Note that $z$ is a central element of $\h_{2n+1}(\C)$, thus for $1\leqslant i,j\leqslant n$, we have
\begin{eqnarray*}
[D(x_{i}),D(y_{j})] & = & [a_{i}z+ax_{i}+cy_{i},b_{j}z+bx_{j}+dy_{j}] \\
 & = & [ax_{i},bx_{j}]+[ax_{i},dy_{j}]+[cy_{i},bx_{j}]+[cy_{i},dy_{j}]\\
 &=&\delta_{ij}\cdot ad\cdot z-\delta_{ji}\cdot bc\cdot z\\
 &=&\delta_{ij}\cdot \det(\theta)\cdot z
\end{eqnarray*}
and $D([x_{i},y_{j}])=D(\delta_{ij}\cdot z)=\delta_{ij}\cdot  \det(\theta)\cdot z=[D(x_{i}),D(y_{j})]$.
Moreover, using the fact again that $z$ is a central element of $\h_{2n+1}(\C)$ we see that
$D([x,z])=0=\det(\theta)\cdot[D(x),z]=[D(x),D(z)]$ for each $x\in\{x_{1},\dots,x_{n},y_{1},\dots,y_{n}\}$. Thus
$D$ is an algebra homomorphism.

Now to prove that $D\in \MLie(\h_{2n+1}(\C))$, it is sufficient to show that the Hom-Jacobi identity follows for $D$. For arbitrary $x,y,w\in\{z,x_{1},\dots,x_{n},y_{1},\dots,y_{n}\}$,  the generating relations (\ref{eq4.1}) imply that
$[y,w], [w,x]\textrm{ and  }[x,y]\in \C\cdot z$ are central elements of $\h_{2n+1}(\C)$. Hence,
$[D(x),[y,w]]=[D(y),[w,x]]=[D(w),[x,y]]=0$ and
$[D(x),[y,w]]+[D(y),[w,x]]+[D(w),[x,y]]=0.$
Therefore, the Hom-Jacobi identity for $D$ follows; and $D\in \MLie(\h_{2n+1}(\C))$ as desired.
\end{proof}

\begin{coro}\label{coro4.3}
There exists a nontrivial involutive Hom-Lie algebra structure on $\h_{2n+1}(\C)$.
\end{coro}

\begin{proof}
Consider $D=D(-1,0,0,-1;0)\in \MLie(\h_{2n+1}(\C))$. Then $D^{2}=I_{2n+1}$ and $D\neq I_{2n+1}$. Thus it gives rise to a
nontrivial involutive Hom-Lie algebra structure on $\h_{2n+1}(\C)$.
\end{proof}

\begin{coro}
There exists a non-involutive nontrivial  regular Hom-Lie algebra structure on $\h_{2n+1}(\C)$. As a result,
$\ILie(\h_{2n+1}(\C))$ is strictly contained in $\RLie(\h_{2n+1}(\C))$.
\end{coro}

\begin{proof}
Let $D=D(1,1,-1,1;\alpha)$ with $\alpha$ arbitrary. Then $\det(D)=2^{n+1}$ and $D^{2}\neq I_{2n+1}$. Hence, $D\in \RLie(\h_{2n+1}(\C))\setminus \ILie(\h_{2n+1}(\C))$ is a
nontrivial non-involutive Hom-Lie algebra structure on $\h_{2n+1}(\C)$.
\end{proof}

\begin{coro}\label{coro4.5}
$\RLie(\h_{2n+1}(\C))$ is strictly contained in $\MLie(\h_{2n+1}(\C))$.
\end{coro}

\begin{proof}
The element $D=D(1,1,1,1;\alpha)\in \MLie(\h_{2n+1}(\C))\setminus \RLie(\h_{2n+1}(\C))$ is a
nontrivial non-regular Hom-Lie algebra structure on $\h_{2n+1}(\C)$.
\end{proof}

\begin{coro}
The affine variety $\MLie(\h_{2n+1}(\C))$ has dimension at least $2n+4$.
\end{coro}

\begin{proof}
The affine variety consisting of all $D(a,b,c,d;\alpha)$, which is contained in $\MLie(\h_{2n+1}(\C))$ by Proposition \ref{prop4.2},  has dimension $2n+4$. Thus $\dim(\MLie(\h_{2n+1}(\C))\geqslant 2n+4$.
\end{proof}

\subsection*{Upper triangular Lie algebras}
Let $\uu_{n}(\C)$ be the Lie algebra of upper triangular matrices over $\C$. We may take the standard basis
$\{E_{ij}\mid 1\leqslant i\leqslant j\leqslant n\}$ for $\uu_{n}(\C)$, where $E_{ij}$ denotes the $n\times n$-matrix in which
the $(i,j)$-entry is 1 and zero otherwise. Then $\dim(\uu_{n}(\C))=n(n+1)/2$ and there are nontrivial relations among the generators $E_{ij}$:
\begin{equation}
\label{ }
[E_{ij},E_{k\ell}]=\delta_{jk}\cdot E_{i\ell}-\delta_{\ell i}\cdot E_{kj}.
\end{equation}
Using the method in Section \ref{sec2}, here we give a complete description on multiplicative Hom-Lie algebras on the  Lie algebras $\uu_{2}(\C)$ and $\uu_{3}(\C)$ with a sketch of proofs, without going into detailed proofs.

\begin{thm}\label{thm4.7}
The affine variety $\MLie(\uu_{2}(\C))$ can be decomposed into three 4-dimensional irreducible components $\V(\p_{1}),\V(\p_{2})$ and $\V(\p_{3})$, where
\begin{eqnarray*}
\V(\p_{1}) & = & \left\{\begin{pmatrix}
      a& 0&c   \\
      b&  0&d\\
      a&0&c
\end{pmatrix}\Bigg|~ a,b,c,d\in\C\right\}, \\
\V(\p_{2}) & = & \left\{\begin{pmatrix}
      a& 0&d  \\
      b&  c&-b\\
      a-1&0&d+1
\end{pmatrix}\Bigg|~ a,b,c,d\in\C\right\}, \\
\V(\p_{3}) & = & \left\{\begin{pmatrix}
      a& 0&d  \\
      b&  0&-b\\
      c&0&a-c+d
\end{pmatrix}\Bigg|~ a,b,c,d\in\C\right\}.
\end{eqnarray*}
\end{thm}

\begin{proof}[Sketch of Proof]
Since $\uu_{2}(\C)$ is a 3-dimensional Lie algebra, we embed $\MLie(\uu_{2}(\C))$ into an affine variety in the 9-dimensional affine space by expressing each element in $\MLie(\uu_{2}(\C))$ as a $3\times 3$-matrix over $\C$. Consider the ideal $I$ of 
the polynomial algebra $\C[x_{ij}\mid 1\leqslant i,j\leqslant 3]$ generated by 
$\{x_{11} + x_{13} - x_{31} - x_{33}, x_{12}, x_{13}x_{21} + x_{13}x_{23} - x_{21}x_{33} - x_{23}x_{33}, x_{13}x_{22} - x_{22}x_{33} + x_{22}, x_{21}x_{22} + x_{22}x_{23}, x_{32}\}$. This set is actually a Gr\"obner basis for $I$ with respect to the lexicographic ordering: $x_{11}>x_{12}>x_{13}> x_{21}>x_{22}>x_{23}>x_{31}> x_{32}> x_{33}$. We use the same strategy in Remark \ref{rem2.2} to prove that $\MLie(\uu_{2}(\C))=\V(I)$. First of all, a direct calculation shows that $\MLie(\uu_{2}(\C))$ is contained in $\V(I)$. Secondly, we need to decompose the ideal $I$, and the similar method in Section 2 can be applied to conclude that $I$ has three irreducible components: $\p_1,\p_2$, and $\p_3$, where
\begin{eqnarray*}
\p_1 & = & \ideal {x_{11} - x_{31}, x_{12}, x_{13} - x_{33}, x_{22}, x_{32}} \\
\p_2 & = & \ideal{x_{11} - x_{31} - 1, x_{12}, x_{13} - x_{33} + 1, x_{21} + x_{23}, x_{32}}\\ 
\p_3 & = & \ideal{x_{11} + x_{13} - x_{31} - x_{33}, x_{12}, x_{21} + x_{23}, x_{22}, x_{32}}.
\end{eqnarray*}
Now it is not difficult to see that $\p_1,\p_2$, and $\p_3$ are prime ideals and elements in each $\V(\p_i)$ are of the matrix form in the statement. One may show that $\V(I)=\V(\p_1)\cup\V(\p_2)\cup\V(\p_3)$, which together with a direct verification that all elements in $\V(\p_1)\cup\V(\p_2)\cup\V(\p_3)$ give rise to a multiplicative Hom-Lie algebra structure on $\uu_{2}(\C)$, forces that $\V(I)\subseteq\MLie(\uu_{2}(\C))$. Therefore, $\MLie(\uu_{2}(\C))=\V(I)$.
\end{proof}

A similar strategy can be applied to describe the variety $\MLie(\uu_{3}(\C))$. Note that $\dim_\C(\uu_{3}(\C))$ is equal to 6.
Thus each element in $\MLie(\uu_{3}(\C))$ can be expressed by a $6\times 6$-matrix over $\C$, and we need to work in the
36-dimensional affine space. Suppose that the corresponding polynomial ring we need is $\C[x_{ij}\mid 1\leqslant i,j\leqslant 6]$ and denote by $I$ the vanishing ideal of  $\MLie(\uu_{3}(\C))$. We will give a Gr\"obner basis for $I$ with respect to the lexicographic ordering: $x_{11}>x_{12}>\cdots>x_{66}$ as previously, and decompose $I$ into three irreducible components 
 $\p_1,\p_2$, and $\p_3$. To describe elements in $\V(\p_i)$, we first define
$$P_{a,b,c,d,e,f,g}:=\begin{pmatrix}
    a &0 &0&d&0& f  \\
    0 &0 &0&0&0& 0  \\
    b&0 &0&e&0& g  \\
   a&0 &0&d&0& f  \\
    c&0 &0&-c&0& 0  \\
   a& 0&0&d&0& f  \\
\end{pmatrix},\quad Q_{a,b,c,d,e,f,g}:=\begin{pmatrix}
    a &0 &0&c&0& f  \\
    0 &0 &0&d&0& -d  \\
    b&0 &0&e&0& g  \\
   a&0 &0&c&0& f  \\
    0&0 &0&0&0& 0  \\
   a& 0&0&c&0& f  \\
\end{pmatrix}$$
and
$$T_{a,b,c,d,e}:=\begin{pmatrix}
    a &0 &0&c&0& d \\
    0 &e &0&0&0& 0  \\
    b&0 &e&0&0& -b  \\
   a-1&0 &0&c+1&0& d  \\
    0&0 &0&0&1& 0  \\
   a-1& 0&0&c&0& d+1  \\
\end{pmatrix}.$$

\begin{thm}\label{thm4.8}
The affine variety $\MLie(\uu_{3}(\C))$ can be decomposed into two 7-dimensional irreducible
components $\V(\p_{1}), \V(\p_{2})$ and one 5-dimensional irreducible component $\V(\p_{3})$, where
\begin{eqnarray*}
\V(\p_{1}) & = & \left\{P_{a,b,c,d,e,f,g}\mid a,b,c,d,e,f,g\in\C\right\}, \\
\V(\p_{2}) & = & \left\{Q_{a,b,c,d,e,f,g}\mid a,b,c,d,e,f,g\in\C\right\}, \\
\V(\p_{3}) & = & \left\{T_{a,b,c,d,e}\mid a,b,c,d,e\in\C\right\}.
\end{eqnarray*}
\end{thm}

\begin{proof}[Sketch of Proof] The method of proving this statement is basically the same as the proof sketch in Theorem \ref{thm4.7}. One of the essential points is to determine the generating set for  the vanishing ideal $I$. Here we define $I$ to be the ideal generated by $\cup_{i=1}^6\B_i$, where 
\begin{eqnarray*}
\B_1&:=&\{x_{11} - x_{55} - x_{61}, x_{12}, x_{13}, x_{14} - x_{64},x_{15}, x_{16} + x_{55} - x_{66}\}\\
\B_2&:=&\{x_{21}, x_{22} - x_{33}, x_{23}, x_{24} + x_{26}, x_{25},x_{26}x_{33}, x_{26}x_{54},x_{26}x_{55}\}\\
\B_3&:=&\{ x_{31}x_{33} + x_{33}x_{36},x_{31}x_{55} + x_{36}x_{55}, x_{32}, x_{33}x_{34}, x_{33}x_{54}, x_{33}x_{55} - x_{33}, x_{34}x_{55},x_{35}\}\\
\B_4&:=&\{x_{41} - x_{61},x_{42},x_{43},x_{44} - x_{55} - x_{64},x_{45},x_{46} + x_{55} - x_{66}\}\\
\B_5&:=&\{x_{51} + x_{54},x_{52},x_{53},x_{54}x_{55},x_{55}^2 - x_{55},x_{56}\}
\end{eqnarray*}
 and $\B_6:=\{x_{62},x_{63},x_{65}\}$. The union of these $\B_i$ is a Gr\"obner basis for $I$. Another key point is to determine the generators of the prime ideals $\p_1,\p_2$, and $\p_3$. In this case,  we can see from the matrix form of elements of $\V(\p_i)$ that they are generated by polynomials of degree 1. More precisely, $\p_1$ is generated by $\{x_{11} - x_{61},x_{12}, x_{13},x_{14} - x_{64},x_{15},x_{16} - x_{66}, x_{2i}, x_{32},x_{33},x_{35},x_{41} - x_{61},x_{42},
 x_{43},x_{44} - x_{64},x_{45},x_{46} - x_{66},x_{51} + x_{54},x_{52},x_{53},x_{55},x_{56},x_{62},x_{63},x_{65}\mid 1\leqslant i\leqslant 6\}$, $\p_2$ is generated by $\{x_{11} - x_{61}, x_{12}, x_{13}, x_{14} - x_{64}, x_{15}, x_{16} - x_{66}, x_{21}, x_{22},
 x_{23},x_{24} + x_{26},x_{25},x_{32}, x_{33},x_{35}, x_{41} - x_{61},x_{42},x_{43},$ $x_{44} - x_{64},x_{45},x_{46} - x_{66},$ $x_{51},x_{52},x_{53},x_{54},x_{55},x_{56},x_{62},x_{63},x_{65}\}$, and the ideal $\p_3$ is generated by 
 $\{x_{11} - x_{61},x_{12},x_{13},x_{14} - x_{64},x_{15},x_{16} - x_{66}, x_{2i},x_{32},x_{33},x_{35}, x_{41} - x_{61},x_{42},
x_{43},x_{44} - x_{64},x_{45},x_{46} - x_{66},x_{51} + x_{54},x_{52},x_{53},x_{55},x_{56},$ $x_{62},x_{63},x_{65}\mid 1\leqslant i\leqslant 6\}$. Gathering this information together and following the strategy  in Remark \ref{rem2.2} (or Theorem \ref{thm4.7}), one can prove the statement. 
\end{proof}

\begin{rem}
{\rm
Considering Theorem \ref{thm3.3} and the results of Theorems \ref{thm4.7} and \ref{thm4.8},  an open problem is to determine whether the affine variety $\MLie(\uu_{n}(\C))$ (for all $n\geqslant 2$) always has three irreducible components. 
 \hbo
}\end{rem}

\section{Derivations of Multiplicative Hom-Lie Algebras}\label{sec5}
\setcounter{equation}{0}
\renewcommand{\theequation}
{5.\arabic{equation}}
\setcounter{theorem}{0}
\renewcommand{\thetheorem}
{5.\arabic{theorem}}

\noindent Let $\g$ be a finite-dimensional complex Lie algebra and  $D\in\MLie(\g)$ be a  multiplicative
Hom-Lie algebra on $\g$. Let $k\in\N$ and recall that a linear transformation $\delta:\g\ra\g$ is called a $D^{k}$-\textit{derivation} of $D$ if  $D\circ \delta=\delta\circ D$ and $\delta([x,y])=[\delta(x),D^{k}(y)]+[D^{k}(x),\delta(y)]$
for every $x,y\in\g$. Here $D^{k}$ denotes the composite map of $k$ copies of $D$, with the convention that $D^{0}:=I_{\g}$ and $D^{1}:=D.$
We denote by $\Der_{k}(\g)$ the space of all $D^{k}$-derivations of $D$ and call
$$\Der_{D}(\g):=\bigoplus_{k=0}^{\infty}\Der_{k}(\g)$$
the \textit{derivation algebra} of $(\g,D)$, which is a Lie algebra with the bracket product:
$$[\delta,\tau]:=\delta\circ\tau-\tau\circ\delta.$$
Note that $[\delta,\tau]\in \Der_{k+s}(\g)$ for $\delta\in \Der_{k}(\g)$ and $\tau\in \Der_{s}(\g)$; see \cite[Section 3]{She12} for details.
We define the \textit{Hilbert series} of $\Der_{D}(\g)$ to be the formal series:
$$\HH(\Der_{D}(\g),t):=\sum_{k=0}^{\infty}\dim_{\C}(\Der_{k}(\g))\cdot t^{k}$$
where $t$ denotes a real indeterminant. In order to make the geometric series $\sum_{k=0}^\infty t^k$ convergent,  we need to assume that $|t|<1$, except for the second statement of Theorem \ref{thm5.5} where we assume that $|t^{m-1}|<1$.
Let $\Der(\g)$ be the usual derivation algebra of $\g$.
We start this section with the following two general properties.

\begin{prop}
The space $\Der_{0}(\g)$ is a Lie subalgebra of $\Der(\g)$.
\end{prop}

\begin{proof}
As the space $\Der_{0}(\g)$ consists of all derivations of $\g$ that commute with $D$, it follows that
$\Der_{0}(\g)$ is a subspace of $\Der(\g)$. Taking arbitrary $\delta,\tau\in \Der_{0}(\g)$, we have
$[\delta,\tau]\circ D=(\delta\circ\tau-\tau\circ\delta)\circ D=D\circ (\delta\circ\tau-\tau\circ\delta)=D\circ[\delta,\tau]$.
Thus $\Der_{0}(\g)$ is a Lie algebra.
\end{proof}

\begin{prop}
The left multiplication with $D$ gives rise to a linear map
$$\rho_{D}^{k}:\Der_{k}(\g)\ra \Der_{k+1}(\g),\quad \delta\mapsto D\circ \delta$$
for every $k\in \N$.
\end{prop}

\begin{proof}
It is immediate to see that $\rho_{D}^{k}$ is  linear. To complete the proof, it suffices to show that
$D\circ \delta\in \Der_{k+1}(\g)$ for each $\delta\in \Der_{k}(\g)$. Indeed,
$D\circ(D\circ \delta)-(D\circ \delta)\circ D=D\circ(D\circ \delta)-D\circ (\delta\circ D)=
D\circ(D\circ \delta)-D\circ (D\circ\delta)=0$; thus $D$ and $D\circ \delta$ are commutes.  For every $x,y\in\g$, since
$\delta([x,y])=[\delta(x),D^{k}(y)]+[D^{k}(x),\delta(y)]$ and $D$ is an algebra homomorphism, it follows that
$$D\circ\delta([x,y])=[D\circ\delta(x),D^{k+1}(y)]+[D^{k+1}(x),D\circ\delta(y)].$$
Hence, $\rho_{D}^{k}$ is a linear map.
\end{proof}

\begin{coro}\label{coro5.3}
If $D$ is invertible, then $\rho_{D}^{k}$ is a linear isomorphism.
\end{coro}

\begin{proof}
Clearly, the left multiplication with $D^{-1}$ gives rise to a linear map
$\rho_{D^{-1}}^{k+1}:\Der_{k+1}(\g)\ra \Der_{k}(\g)$ and $(\rho_{D}^{k})^{-1}=\rho_{D^{-1}}^{k+1}$.
\end{proof}

\begin{thm}\label{thm5.4}
Let $D$ be a regular  Hom-Lie algebra on $\g$. Then
$$\HH(\Der_{D}(\g),t)=\frac{\dim_{\C}(\Der_{0}(\g))}{1-t}.$$
In particular, $\dim_{\C}(\Der_{D}(\g))$ is either zero or infinite.
\end{thm}

\begin{proof}
Since $D$ is invertible, it follows from Corollary \ref{coro5.3} that $\dim_{\C}(\Der_{0}(\g))=\dim_{\C}(\Der_{1}(\g))=\dim_{\C}(\Der_{2}(\g))=\cdots$. Suppose $\dim_{\C}(\Der_{0}(\g))=\ell$. Thus
\begin{eqnarray*}
\HH(\Der_{D}(\g),t)&=&\ell\cdot t^{0}+\ell\cdot t^{1}+\ell\cdot t^{2}+\cdots=\ell\cdot(1+t+t^{2}+\cdots)=\frac{\ell}{1-t}.
\end{eqnarray*}
In particular, if $\Der_{0}(\g)=\{0\}$, then $\dim_{\C}(\Der_{D}(\g))=0$; if  $\Der_{0}(\g)\neq\{0\}$, then $\dim_{\C}(\Der_{D}(\g))$ is infinite.
\end{proof}

\begin{thm}\label{thm5.5}
Let $D$ be a Hom-Lie algebra on $\g$.
\begin{enumerate}
  \item If $D$ is nilpotent, then there exists a polynomial function $f(t)\in\Z[t]$ such that
$$\HH(\Der_{D}(\g),t)=\frac{f(t)}{1-t}.$$
  \item If $D^{m}=D$ for some integer $m\geqslant 2$, then there also exists a polynomial function $f(t)\in\Z[t]$ such that
$$\HH(\Der_{D}(\g),t)=\frac{f(t)}{1-t^{m-1}}.$$
\end{enumerate}
\end{thm}

\begin{proof} Let $\dim_{\C}(\Der_{k}(\g))=\ell_{k}$ for $k\in\N^{+}$.
We first consider the case where $D$ is nilpotent, i.e., $D^{n}=0$ for some $n\in\N^{+}$. Then $\Der_{k}(\g)=\{\delta\in\gl(\g)\mid D\circ\delta=\delta\circ D \textrm{ and } \delta([x,y])=0\textrm{ for all }x,y\in\g\}$ for $k\geqslant n$. Thus
\begin{eqnarray*}
\HH(\Der_{D}(\g),t)&=&\ell_{0}+\ell_{1}\cdot t+\cdots+\ell_{n-1}\cdot t^{n-1}+\ell_{n}\cdot t^{n}+\ell_{n}\cdot t^{n+1}+\cdots\\
&=&\ell_{0}+\ell_{1}\cdot t+\cdots+\ell_{n-1}\cdot t^{n-1}+\ell_{n}\cdot t^{n}\cdot(1+t+t^{2}+\cdots)\\
&=&\ell_{0}+\ell_{1}\cdot t+\cdots+\ell_{n-1}\cdot t^{n-1}+\frac{\ell_{n}\cdot t^{n}}{1-t}.
\end{eqnarray*}
Let $f(t)=(1-t)(\sum_{k=0}^{n-1}\ell_{k}\cdot t^{k})+\ell_{n}\cdot t^{n}$. Then $\HH(\Der_{D}(\g),t)=f(t)/(1-t).$

For the second case, we observe that $\Der_{k}(\g)=\Der_{s}(\g)$ if $k-s\equiv 0\mod (m-1)$ for all $k,s>0$. Thus
\begin{eqnarray*}
\HH(\Der_{D}(\g),t)
&=&\ell_{0}+\ell_{1}\cdot t+\cdots+\ell_{m-1}\cdot t^{m-1}+\ell_{1}\cdot t^{m}+\ell_{2}\cdot t^{m+1}+\cdots\\
&=&\ell_{0}+\ell_{1}\cdot t\cdot(1+t^{m-1}+t^{2(m-1)}+\cdots)+\ell_{2}\cdot t^{2}\cdot(1+t^{m-1}+t^{2(m-1)}+\cdots)\\
&&+\cdots+ \ell_{m-1}\cdot t^{m-1}\cdot(1+t^{m-1}+t^{2(m-1)}+\cdots)\\
&=&\ell_{0}+\frac{\ell_{1}\cdot t}{1-t^{m-1}}+\frac{\ell_{2}\cdot t^{2}}{1-t^{m-1}}+\cdots+\frac{\ell_{m-1}\cdot t^{m-1}}{1-t^{m-1}}.
\end{eqnarray*}
Let $f(t)=\ell_{0}\cdot(1-t^{m-1})+\sum_{k=1}^{m-1} \ell_{k}\cdot t^{k}$. Then $\HH(\Der_{D}(\g),t)=f(t)/(1-t^{m-1}).$
\end{proof}

 The following example illustrates how to use the method mentioned above to explicitly
calculate the Hilbert series of the derivation algebra of a Hom-Lie algebra.

\begin{exam}{\rm
We consider the multiplicative Hom-Lie algebra $C_{a}$ on $\gl_{2}(\C)$ appeared in Corollary \ref{coro2.4}. We observe that
$$C_{a}^{k}=\begin{pmatrix}
           2^{k-1}a^{k}& 0&0& 2^{k-1}a^{k}\\
           0& 0&0& 0\\
            0& 0&0& 0\\
            2^{k-1}a^{k}& 0&0& 2^{k-1}a^{k}\\
         \end{pmatrix}.$$
Thus $C_{1/2}^{k}=C_{1/2}$ for $k>1$ and so in order to give an explicit formula for the Hilbert series $\HH(\Der_{C_{1/2}}(\gl_{2}(\C),t))$,
it is sufficient to examine the dimensions of $\Der_{0}(\gl_{2}(\C))$ and $\Der_{1}(\gl_{2}(\C))$ for $D=C_{1/2}$.
A long but direct calculation shows that the usual derivation algebra $\Der(\gl_{2}(\C))$ is a 4-dimensional irreducible affine variety (also a 4-dimensional vector space over $\C$), consisting of all derivations with the following forms:
$$\begin{pmatrix}
     e & c&b&e   \\
      -b&-d  &0&b\\
      -c&0&d&c\\
      e&-c&-b&e
\end{pmatrix}$$
where $b,c,d,e\in\C$. Now it is easy to check that each element of $\Der(\gl_{2}(\C))$ commutes with $C_{a}$.
Hence, $\Der_{0}(\gl_{2}(\C))=\Der(\gl_{2}(\C))$. A direct calculation shows that $\Der_{1}(\gl_{2}(\C))=\V(\p_{1})$, where
$\p_{1}$ is defined as in Corollary \ref{coro2.4}. Therefore,
$$\HH(\Der_{C_{1/2}}(\gl_{2}(\C),t))=4+t+t^{2}+\cdots=\frac{4-3t}{1-t}$$
is a rational function. \hbo
}\end{exam}

We close this paper with a remark that might give readers a hint to use the method in other possible related directions.

\begin{rem}{\rm
It has been shown recently that the method of our article is useful to deal with some linear structures on Lie algebras; see
\cite[Section 5]{CCZ21}. Except for Hom-Lie algebras, we also note that there exists a relatively new notion of nonassociative algebras, $\omega$-Lie algebras, that contain Lie algebras as a subclass and have attracted many researchers' attention; see for example, \cite{CLZ14, CZ17,CZZZ18, Zha21}. Our method appeared in this paper might be applied to studying these $\omega$-algebra structures.
\hbo}\end{rem}


\raggedright

\begin{bibdiv}
\begin{biblist}


\bib{AC19}{article}{
   author={Alvarez, Mar\'{\i}a Alejandra},
   author={Cartes, Francisco},
   title={Cohomology and deformations for the Heisenberg Hom-Lie algebras},
   journal={Linear Multilinear Algebra},
   volume={67},
   date={2019},
   number={11},
   pages={2209--2229},
}


\bib{Bau99}{article}{
   author={Baues, Oliver},
   title={Left-symmetric algebras for $\germ{gl}(n)$},
   journal={Trans. Amer. Math. Soc.},
   volume={351},
   date={1999},
   number={7},
   pages={2979--2996},
}



\bib{BM14}{article}{
   author={Benayadi, Sa\"{\i}d},
   author={Makhlouf, Abdenacer},
   title={Hom-Lie algebras with symmetric invariant nondegenerate bilinear
   forms},
   journal={J. Geom. Phys.},
   volume={76},
   date={2014},
   pages={38--60},
}

\bib{CCZ21}{article}{
   author={Chang, Hongliang},
   author={Chen, Yin},
   author={Zhang, Runxuan},
   title={A generalization on derivations of Lie algebras},
   journal={Electron. Res. Arch.},
   volume={29},
   date={2021},
   number={3},
  pages={2457--2473},
}


\bib{CLZ14}{article}{
   author={Chen, Yin},
   author={Liu, Chang},
   author={Zhang, Runxuan},
   title={Classification of three-dimensional complex $\omega$-Lie algebras},
   journal={Port. Math.},
   volume={71},
   date={2014},
   number={2},
   pages={97--108},
   issn={0032-5155},
}

\bib{CZZZ18}{article}{
   author={Chen, Yin},
   author={Zhang, Ziping},
   author={Zhang, Runxuan},
   author={Zhuang, Rushu},
   title={Derivations, automorphisms, and representations of complex
   $\omega$-Lie algebras},
   journal={Comm. Algebra},
   volume={46},
   date={2018},
   number={2},
   pages={708--726},
   issn={0092-7872},
}

\bib{CZ17}{article}{
   author={Chen, Yin},
   author={Zhang, Runxuan},
   title={Simple $\omega$-Lie algebras and 4-dimensional $\omega$-Lie
   algebras over $\Bbb{C}$},
   journal={Bull. Malays. Math. Sci. Soc.},
   volume={40},
   date={2017},
   number={3},
   pages={1377--1390},
   issn={0126-6705},
}

\bib{CLO07}{book}{
   author={Cox, David},
   author={Little, John},
   author={O'Shea, Donal},
   title={Ideals, varieties, and algorithms},
   series={Undergraduate Texts in Mathematics},
   edition={3},
   publisher={Springer, New York},
   date={2007},
}


\bib{DK15}{book}{
   author={Derksen, Harm},
   author={Kemper, Gregor},
   title={Computational invariant theory},
   series={Encyclopaedia of Mathematical Sciences},
   volume={130},
   edition={Second enlarged edition},
   publisher={Springer, Heidelberg},
   date={2015},
}



\bib{GDSSV20}{article}{
   author={Garc\'{\i}a-Delgado, R.},
   author={Salgado, G.},
   author={S\'{a}nchez-Valenzuela, O. A.},
   title={On 3-dimensional complex Hom-Lie algebras},
   journal={J. Algebra},
   volume={555},
   date={2020},
   pages={361--385},
}

\bib{GK96}{book}{
   author={Goze, M.},
   author={Khakimdjanov, Y.},
   title={Nilpotent Lie algebras},
   series={Mathematics and its Applications},
   volume={361},
   publisher={Kluwer Academic Publishers Group, Dordrecht},
   date={1996},
}
	


\bib{HLS06}{article}{
   author={Hartwig, Jonas T.},
   author={Larsson, Daniel},
   author={Silvestrov, Sergei D.},
   title={Deformations of Lie algebras using $\sigma$-derivations},
   journal={J. Algebra},
   volume={295},
   date={2006},
   number={2},
   pages={314--361},
}

\bib{JL08}{article}{
   author={Jin, Quanqin},
   author={Li, Xiaochao},
   title={Hom-Lie algebra structures on semi-simple Lie algebras},
   journal={J. Algebra},
   volume={319},
   date={2008},
   number={4},
   pages={1398--1408},
}

\bib{MS08}{article}{
   author={Makhlouf, Abdenacer},
   author={Silvestrov, Sergei D.},
   title={Hom-algebra structures},
   journal={J. Gen. Lie Theory Appl.},
   volume={2},
   date={2008},
   number={2},
   pages={51--64},
}


\bib{NS02}{book}{
   author={Neusel, Mara D.},
   author={Smith, Larry},
   title={Invariant theory of finite groups},
   series={Mathematical Surveys and Monographs},
   volume={94},
   publisher={American Mathematical Society, Providence, RI},
   date={2002},
}



\bib{OOS19}{article}{
   author={Ongong'a, E.},
   author={Ongaro, J.},
   author={Silvestrov, S.},
   title={Hom-Lie structures on complex 4-dimensional Lie algebras},
   note={International Workshop on Lie Theory and Its Applications in Physics},
   pages={373--381},
   publisher={Springer, Singapore},
   date={2019},
}




\bib{Rem18}{article}{
   author={Remm, Elisabeth},
   title={3-dimensional skew-symmetric algebras and the variety of Hom-Lie
   algebras},
   journal={Algebra Colloq.},
   volume={25},
   date={2018},
   number={4},
   pages={547--566},
}

\bib{She12}{article}{
   author={Sheng, Yunhe},
   title={Representations of Hom-Lie algebras},
   journal={Algebr. Represent. Theory},
   volume={15},
   date={2012},
   number={6},
   pages={1081--1098},
}

\bib{Vas98}{book}{
   author={Vasconcelos, Wolmer V.},
   title={Computational methods in commutative algebra and algebraic
   geometry},
   series={Algorithms and Computation in Mathematics},
   volume={2},
   publisher={Springer-Verlag, Berlin},
   date={1998},
}

\bib{XJL15}{article}{
   author={Xie, Wenjuan},
   author={Jin, Quanqin},
   author={Liu, Wende},
   title={${\rm Hom}$-structures on semi-simple Lie algebras},
   journal={Open Math.},
   volume={13},
   date={2015},
   number={1},
   pages={617--630},
}



\bib{Yau11}{article}{
   author={Yau, Donald},
   title={Hom-Novikov algebras},
   journal={J. Phys. A},
   volume={44},
   date={2011},
   number={8},
   pages={085202, 20},
}

\bib{Zha21}{article}{
   author={Zhang, Runxuan},
   title={Representations of $\omega$-Lie algebras and tailed derivations of Lie algebras},
   journal={Internat. J. Algebra Comput.},
  volume={31},
   date={2021},
   number={2},
  pages={325--339},
}

\bib{ZHB11}{article}{
   author={Zhang, Runxuan},
   author={Hou, Dongping},
   author={Bai, Chengming},
   title={A Hom-version of the affinizations of Balinskii-Novikov and
   Novikov superalgebras},
   journal={J. Math. Phys.},
   volume={52},
   date={2011},
   number={2},
   pages={023505, 19},
}


 \end{biblist}
\end{bibdiv}
\end{document}